\documentclass[12pt]{article}
\usepackage{amssymb,amsfonts,amsmath,amsthm,color}
\usepackage{color}
\usepackage[colorlinks = true,
linkcolor = red,
urlcolor  = red,
citecolor = red,
anchorcolor = red]{hyperref}
\usepackage{indentfirst}

\newtheorem{Assumption}{Assumption}[section]
\newtheorem{theorem}{Theorem}[section]
\newtheorem{lemma}[theorem]{Lemma}
\newtheorem{proposition}[theorem]{Proposition}

\newtheorem{definition}{Definition}[section]

\numberwithin{equation}{section}

\begin{document}
	
	\title{\bf Dynamics of Fractional Wave Equations with Nonlocal Damping}
	
	\author{\small\bf Raúl E. Vidal\thanks{E-mail: \texttt{raul.vidal@unc.edu.ar}}\\
		{\small FAMAF, 
			Universidad Nacional de C\'ordoba,
			CIEM (CONICET) 
		}\\
		{\small 5000 C\'ordoba, Argentina.}\medskip\\
		{\small {\bf Vando Narciso}\thanks{Corresponding author. E-mail: \texttt{vnarciso@uems.br}. Research supported by FUNDECT/CNPq Grant No.~15/2024 and by the UEMS International Mobility Program (Call No.~42/2025--DRI/UEMS).}} \\
		{\small Center of Exact and Technological Sciences,} 
		{\small State University of Mato Grosso do Sul }\\
		{\small 79804-970 Dourados, MS, Brazil.} }
	\date{}
	\maketitle
	
	\begin{abstract}
		In this work we study a fractional wave equation with Balakrishnan--Taylor type damping posed on a bounded domain $\Omega\subset\mathbb{R}^n$. The model couples the fractional Laplacian with a nonlinear damping coefficient depending on the fractional energy, leading to a doubly nonlocal evolution equation that extends the classical wave equation with energy-dependent dissipation. By means of the theory of linear operators, we establish the global well-posedness of both mild and regular solutions. We then investigate the long-time dynamics of the associated semigroup and prove that it is gradient and asymptotically smooth. As a consequence, we establish the existence of a compact global attractor and show that it coincides with the unstable manifold of the set of stationary solutions. To the best of our knowledge, this provides the first characterization of the asymptotic dynamics for fractional wave equations with Balakrishnan--Taylor type damping.
	\end{abstract}

	
	\smallskip
	
	

	\section{Introduction}
	In this paper we consider the following fractional wave equation with nonlocal energy-dependent damping:
	\begin{equation}\label{P}
		\begin{cases}
			u_{tt}+(-\Delta)^s u
			+\gamma\bigl([u]_s^2+\|u_t\|_{L^2(\Omega)}^2\bigr)^q u_t
			+f(u)=h(x),
			& \text{in }\Omega\times(0,\infty),\\[0.4em]
			u=0,
			& \text{on }\partial \Omega\times(0,\infty),\\[0.4em]
			u(x,0)=u_0(x),\qquad
			u_t(x,0)=u_1(x),
			& \text{in }\Omega.
		\end{cases}
	\end{equation}
	where $\Omega\subset\mathbb{R}^n$ is a bounded Lipschitz domain,
	$s\in(0,1)$, $\gamma>0$, and $q\ge1/2$. The function
	$h\in L^2(\Omega)$ denotes a time-independent external force, while
	$f:\mathbb{R}\to\mathbb{R}$ is a nonlinear source satisfying suitable growth assumptions specified below. Moreover, $(-\Delta)^s$ denotes the fractional Laplacian of order $s$, defined by
	\begin{align}\label{deffrc}
		(-\Delta)^s u(x)
		=
		2\,
		\mathrm{P.V.}
		\int_{\Omega} 
		\frac{u(x)-u(y)}
		{|x-y|^{n+2s}}
		\,dy.
	\end{align}
	
	The fractional Laplacian is a prime example of a non-local operator that accounts for long-range spatial interactions. Long-jump random walks, often referred to as Lévy flights, are characterized by jump length probability distributions governing the diffusing medium or particles. The singular integral representation in \eqref{deffrc} can be derived as the continuum limit of these discrete, long-jump random walks (see, e.g., \cite{B17, M16, V09} for details). In particular, the fractional Laplacian operator defined in \eqref{deffrc} acts as the infinitesimal generator of a symmetric $2s$-stable Lévy stochastic process. As such, it naturally arises in the modeling of non-local phenomena across diverse fields, ranging from biology to mathematical finance.
	
	Accordingly, numerous studies have investigated models involving so-called anomalous diffusion modeled by the fractional Laplacian. The nonlocal nature of this operator has proven advantageous—and often supported by empirical evidence—in applications where its local counterpart falls short. A wide variety of equations incorporating the fractional Laplacian have been proposed and analyzed across diverse physical and biological domains. These include applications in acoustics and compressional wave propagation \cite{HS10, ZH14}, absorption and dispersion in viscoelastic solids \cite{TC14}, lossy media \cite{CH04}, turbulent motion \cite{DSU08}, phase transitions \cite{SV09}, quasi-geostrophic flow \cite{CV-10}, mathematical finance \cite{S07}, water wave dynamics \cite{BV16}, and nonlinear porous medium equations \cite{CV10, V12}. This list is by no means exhaustive, and further background on anomalous diffusion and fractional dynamics can be found in \cite{BRS16, BH04, MK00, P18}.
	
	In their seminal work \cite{CS07}, Caffarelli and Silvestre demonstrated that the fractional Laplacian operator can be interpreted as a Dirichlet-to-Neumann map, thereby relating this nonlocal operator to a local one in an extended space. This characterization unlocked the use of local PDE techniques to establish various regularity results, laying a solid framework that propelled the study of fractional Laplacian problems. Extending this idea to evolutionary problems, Kemppainen,   Sjögren, and Torrea, in \cite{KST15}, showed that the fractional Laplacian can also be realized as a Dirichlet-to-Neumann operator for a degenerate hyperbolic equation specifically, a wave equation subject to a diffusion term that blows up at the initial time. Through this approach, they established wave-type results analogous to the elliptic extensions in \cite{CS07}.
	
	Regarding works that study wave equations involving the fractional Laplacian, several notable contributions should be mentioned. D'Abbicco and Ebert \cite{AE17} investigated the global existence of solutions for semilinear evolution equations with fractional structural damping, where the fractional Laplacian appears in the first-order time derivative. For Kirchhoff-type fractional wave problems, Pan et al. \cite{PPXZ19} analyzed a model with nonlinear damping and source terms; by combining the Galerkin method with potential well theory, they established global existence, vacuum isolation, asymptotic behavior, and finite-time blow-up of solutions. In a similar framework, Lin et al. \cite{LTXZ20} focused on the blow-up behavior and estimates for the blow-up time under arbitrary positive initial energy. Furthermore, Xiang and Hu \cite{XH21} addressed a viscoelastic fractional wave equation of Kirchhoff type, proving the local and global existence of solutions via Galerkin approximations and investigating global nonexistence through blow-up analysis. Finally, Yang and Zhou \cite{YZ22} studied the long-time behavior of degenerate fractional Kirchhoff wave equations with structural or strong damping, proving the existence of global attractors and the asymptotic compactness of solutions in fractional energy spaces.
	
	Within the broad class of dissipative dynamical systems, wave equations with nonlinear damping have received considerable attention because of their relevance to physical models and the analytical difficulties arising in their mathematical treatment. In this framework, damping mechanisms depending on the total energy of the system---commonly referred to as nonlocal or Krasovskii-type damping---have become an active topic of research. Such a feedback law was originally introduced by Krasovskii~\cite{Krasovskii} for ordinary differential equations. A representative example is
	\begin{align*}
		\ddot{x}(t)+x(t)+\kappa(E(t))\dot{x}(t)=0,
	\end{align*}
	where
	$
	E(t)=\frac12\bigl(x^2(t)+\dot{x}^2(t)\bigr)
	$
	denotes the total mechanical energy. Subsequently, Balakrishnan~\cite{Balakrishnan} and Balakrishnan--Taylor~\cite{Balakrishnan-Taylor} incorporated this type of damping into models arising in flight dynamics and structural engineering. In the Balakrishnan--Taylor framework, the damping coefficient is typically chosen in the form $\kappa(s)=\gamma s^q,$ $\gamma,q>0.$

	An abstract infinite-dimensional formulation of the Krasovskii model was later proposed by Chueshov~\cite[Section 5.3.3]{chueshov-2015} in a separable Hilbert space $H$. More precisely, the evolution problem
	\begin{align}\label{model-chueshov}
		u_{tt}+Au+\kappa(\|A^{1/2}u\|^2+\|u_t\|^2)u_t=0,
		\quad
		u(0)=u_0,\;\; u_t(0)=u_1,
	\end{align}
	was considered, where $A$ is a positive self-adjoint operator with compact resolvent, while $\kappa$ is assumed to be a bounded Lipschitz continuous function on $\mathbb{R}^{+}$ satisfying $\kappa(1)=0$ and strictly increasing for $s>1$. The nonlocal character of the damping, together with its dependence on the instantaneous energy level, creates substantial obstacles in the investigation of the asymptotic behavior of solutions. In particular, the nonlinear dependence of the damping coefficient on the energy makes it difficult to estimate differences of trajectories and to establish asymptotic compactness in the natural phase space. Under these assumptions, Chueshov proved the existence of a \emph{noncompact} global attractor.
	
	The compactness problem was later resolved for a plate equation with Balakrishnan--Taylor type damping. By combining new nonlinear multiplier estimates with the approach introduced by Chueshov and Lasiecka~\cite{chueshov-white}, Gomes \emph{et al.}~\cite{GJLN} established the asymptotic compactness of the associated semigroup and consequently obtained the existence of a \emph{compact} global attractor for a class of Krasovskii-type damping mechanisms. This direction was further pursued by Narciso~\cite{Narciso-Nodea}, who investigated problem \eqref{model-chueshov} with a power-type damping coefficient and an additional source term. Besides proving the existence of a compact global attractor, the author derived estimates for its Kolmogorov $\varepsilon$-entropy. Moreover, assuming suitable nondegeneracy conditions on the damping coefficient, finite-dimensionality and additional regularity properties of the attractor were established. More recently, Zhou and Yang~\cite{Zhou-Yang-2026} generalized these results to a wider family of power-law damping functions.
	
	Nonlocal energy-dependent damping has also been investigated for wave equations. In a bounded domain of $\mathbb{R}^3$, Tang \emph{et al.}~\cite{TYXZ-DCDS} proved the existence of a finite-dimensional global attractor for a wave equation with sublinear Balakrishnan--Taylor damping. The corresponding superlinear case was analyzed by Bezerra \emph{et al.}~\cite{BJN-AA} in the nonautonomous framework, where the existence of a pullback attractor was established through nonlinear multiplier techniques.
	
	For autonomous wave equations posed on the whole space $\mathbb{R}^n$, Freitas and Narciso~\cite{Freitas-narciso} considered a model with superlinear Krasovskii-type damping. They proved the existence of a compact global attractor with respect to the weak topology of
	$
	\mathcal{H}=H^1(\mathbb{R}^n)\times L^2(\mathbb{R}^n),
	$
	by establishing dissipativity together with asymptotic compactness of the corresponding dynamical system. Since Sobolev embeddings fail to be compact in unbounded domains, the proof relied on a compensated compactness argument combined with suitable uniform estimates on the tails of solutions. In addition, the authors obtained the upper semicontinuity of the family of global attractors with respect to the damping parameter.
	
	Motivated by recent developments on dissipative wave equations with energy-dependent damping, we investigate the role of this nonlinear dissipation mechanism in the fractional wave equation \eqref{P}. The coupling of the fractional Laplacian with a Balakrishnan--Taylor type damping gives rise to a doubly nonlocal evolution equation that naturally extends the classical wave model. Moreover, the limiting case $s\to1$ formally recovers the standard wave equation with homogeneous Dirichlet boundary conditions.

	\bigskip
	The main contributions of this paper are summarized as follows.
	
	\begin{itemize}
		\item[(a)] Our first result concerns the well-posedness of problem \eqref{P}. By applying the semigroup theory of linear operators developed by Pazy~\cite{Pazy}, we establish the global existence and uniqueness of both mild (weak) and regular solutions; see Theorem~\ref{theo-existence}.
		
		\item[(b)] Our second and main result concerns the long-time behavior of solutions. We prove that the dynamical system $(\mathcal{H}_s,S_t)$ generated by the mild solutions of \eqref{P} is a gradient dynamical system and is asymptotically smooth. We further show that the set $\mathcal{N}$ of stationary points is bounded. Consequently, by applying the abstract theory of infinite-dimensional dynamical systems developed by Chueshov and Lasiecka~\cite{chueshov-white,chueshov-yellow}, we establish the existence of a compact global attractor $\mathfrak{A}$ in the phase space $\mathcal{H}_s=H^s_0(\Omega)\times L^2(\Omega)$. Moreover, we prove that
		$$\mathfrak{A}=\mathrm{M}^u(\mathcal N),$$
		that is, the global attractor coincides with the unstable manifold of the set of stationary points; see Theorem~\ref{theo-main}.
	\end{itemize}
	
	To the best of our knowledge, this is the first work addressing the long-time dynamics of fractional wave equations with Balakrishnan--Taylor type energy-dependent damping. In particular, we establish the existence of a compact global attractor and characterize it as the unstable manifold of the set of stationary solutions.
	
	\section{Well-posedness}
	\subsection{Functional Spaces and Assumptions}

	\begin{definition} Let $s \in (0,1)$ and let
		\[
		H^{s}(\Omega):=\left\{u \in L^2(\Omega) : \int_{\Omega}\int_{\Omega}\frac{|u(x)-u(y)|^2}{|x-y|^{n+2s}} d x d y < \infty\right\}\]
		be the fractional Sobolev space endowed with the norm 
		\[\|u\|_{H^s(\Omega)}=(\|u\|_{L^2(\Omega)}^2+[u]_{s}^2)^{1/2},\]
		where \[[u]^2_{s}:=\int_{\Omega}\int_{\Omega}\frac{|u(x)-u(y)|^2}{|x-y|^{n+2s}}d xd y,\]
		is a Galiardo seminorm and for every $1\leqslant q \leqslant \infty$, $\|\cdot \|_{L^q(\Omega)}$ is the norm in $L^q(\Omega)$. For notational simplicity, throughout the remainder of this paper, we use the convention
			$\|\cdot\|_{L^q(\Omega)}=\|\cdot\|_q$ and, in particular,
			$\|\cdot\|_{L^2(\Omega)}=\|\cdot\|$.
	\end{definition}
	With this norm, $H^{s}(\Omega)$ is a Hilbert space.  Let $H^s_0(\Omega)$ denote the closure of $C^\infty_0(\Omega)$ in the norm $\|u\|_{H^s(\Omega)}$,     
	It is well known that there exists a constant $K=K(n,s)$ such that for all $u\in H^s_0$, 
	\begin{align}\label{poincare}
		\|u\|\leq K [u]_s.
	\end{align} 
	Consequently, $H^s_0(\Omega)$ can be equipped with the equivalent norm $\|u\|_{H^s_0}=[u]_{s}$, (see Theorem 6.5 and Theorem 6.7 in \cite{DPV}).
	
	When $\Omega=\mathbb{R}^n$, there exists a normalization constant $c(n,s)$ such that, for every $u$ in the Schwartz space $\mathcal{S}(\mathbb{R}^n)$, 
	$$
	(-\Delta)^su(x)=c(n,s)\mathcal{F}^{-1}(|\xi|^{2s}\mathcal{F}u)(x), \qquad \forall x\in \mathbb{R}^n,
	$$
	where $\mathcal{F}$ denotes the Fourier transform and $\mathcal{F}^{-1}$ its inverse, (See Proposition 3.3 in \cite{DPV}). Furthermore, for $u \in \mathcal{S}(\mathbb{R}^n)$, the fractional Laplacian recovers the classical operators in the asymptotic limits $s\to 1^-$ and $s\to 0^+$  (see Proposition 4.4 in \cite{DPV})
	$$
	\lim_{s\to 1^-} (-\Delta )^su(x)=-c(n)\Delta u \quad \text{and} \quad  \lim_{s\to 0^+} (-\Delta )^su(x)=\tilde c(n)  u, 
	$$
	where $\lim_{s \to 1^-} c(n,s) = c(n)$ and $\lim_{s \to 0^+} c(n,s) = \tilde{c}(n)$.
	
	Analogously, if $\Omega\subset\mathbb{R}^n$ is a bounded Lipschitz domain,  the asymptotic behavior of the Gagliardo seminorm for $u\in H^s(\Omega)$ satisfies
	$$
	\lim_{s\to 1^-}(1-s)[u]_s^2=c(n) \|\nabla u\|^2 \quad \text{and} \quad \lim_{s\to 0^+}s[u]_s^2=\tilde c(n) \|u\|^2,
	$$
	for suitable dimensional constants $c(n)$ and $\tilde c(n)$, (see \cite{BBM}).
	
	The study of problem \eqref{P} will be carried out in its natural phase
	space
	$
	\mathcal{H}_s=H_0^s(\Omega)\times L^2(\Omega),
	$
	endowed with the usual inner product
	$$
	\left\langle (u,v),(z,w)\right\rangle_{\mathcal{H}_s}
	=(u,z)_{H_0^s(\Omega)}+(v,w),
	$$
	and the corresponding norm
	$$
	\|(u,v)\|_{\mathcal{H}_s}^{2}
	=\|u\|_{H_0^s(\Omega)}^{2}
	+\|v\|^{2}.
	$$
	It is also known that (see propositions 2.1 and 2.2 in \cite{DPV})
	$$ \|u\|_{H^{s_1}_0(\Omega)} \leq c(n,s_1,s_2)\|u\|_{H^{s_2}_0(\Omega)}, \qquad \mbox{if}\quad 0<s_1\le s_2\leq 1, $$
	where $\|u\|_{H^1_0(\Omega)}=\|\nabla u\|$.
	
	The total energy associated with problem (\ref{P}) is defined by
	\begin{eqnarray}\label{energy-formula}
		E(t)
		=
		\frac{1}{2}\|u(t)\|_{H^s_0(\Omega)}^2+\frac{1}{2}\|u_t(t)\|^2
		+
		\int_{\Omega}F(u(t))\,dx-\int_{\Omega}hudx,
	\end{eqnarray}
	where
	$
	F(r)=\int_{0}^{r}f(\xi)\,d\xi.
	$
	
	The well-posedness of problem \eqref{P} will be established under the
	following assumptions on the nonlinearity $f$.
	
	\begin{Assumption}\label{Assumption}
		Let $f\in C^1(\mathbb{R})$. There exist positive constants
		$C_{f'}$, $C_f$, and
		$c_f\in\left(0,\frac{1}{K^2}\right)$, where $K$ denotes the Poincaré
		constant introduced in \eqref{poincare}, such that
		\begin{equation}\label{assumption_f'}
			|f'(u)|
			\le
			C_{f'}\bigl(1+|u|^{p}\bigr),
			\qquad
			u\in\mathbb{R},
		\end{equation}
		and
		\begin{equation}\label{assumption_f}
			-C_f-\frac{c_f}{2}|u|^2
			\le
			F(u)
			\le
			f(u)u+\frac{c_f}{4}|u|^2,
			\qquad
			u\in\mathbb{R}.
		\end{equation}
		with growth exponent $p$ such that
		\begin{eqnarray}\label{hip-p}
			0<p\le \frac{2s}{n-2s} 
			\quad\mbox{if}\quad n>2s,
			\qquad\mbox{and}\qquad
			p>0
			\quad\mbox{if}\quad n\le 2s.
		\end{eqnarray}
		Condition \eqref{hip-p} implies that, $2(p+1)\leq 2^*_s:=\frac{2n}{n-2s}$ and
		\begin{align*}		H^s_0\hookrightarrow L^{2(p+1)}(\Omega).
		\end{align*}
		(see Theorem 6.5 in \cite{DPV}). 
		
		For convenience, we introduce the constant
		\begin{align*}
			\kappa:=1-c_fK^2>0,
		\end{align*}
		which will be used throughout the paper.

	\end{Assumption}

	\medskip
	\subsection{Cauchy Problem}
	To establish the well-posedness of problem \eqref{P}, we reformulate it as an equivalent first-order evolution equation in the energy space $\mathcal{H}_s$. This formulation enables us to apply the semigroup theory developed by Pazy \cite{Pazy}. Setting $U=(u,v)$ with $v=u_t$, problem \eqref{P} can be rewritten as the following equivalent first-order Cauchy problem:
	\begin{eqnarray}\label{abstrac-cauchy}
		\left\{\begin{array}{l}
			U_t ={A}U +{B}(U), \quad t>0, \medskip \\
			U(0)=(u_0,u_1):=U_0,
		\end{array}\right.
	\end{eqnarray}
	where   ${A}:D(A)\subset\mathcal{H}\to\mathcal{H}$ is a linear  operator defined by
	\begin{equation}\label{def_A}
		AU = \left(\begin{array}{l}
			v \\ -(-\Delta)^s u
		\end{array}\right)^{\perp}, \quad    U\in D(A)= \left\{U\in\mathcal{H}\left|\begin{array}{l}v\in H^s_0(\Omega),\\ (-\Delta)^s u\in L^2(\Omega)\end{array}\right.\right\},
	\end{equation}
	and $B:\mathcal{H}\to\mathcal{H}$ is the nonlinear operator
	\begin{align*}
		B(U)= \left(\begin{array}{l}
			0 \\  -\gamma\|U\|^{2q}_{\mathcal{H}_s}v-f(u)+h
		\end{array}\right)^{\perp}, \quad U=(u,v)\in\mathcal{H}.
	\end{align*}
	
	We are now in a position to state the well-posedness result for the abstract Cauchy problem \eqref{abstrac-cauchy}, and hence for problem \eqref{P}.
	\subsection{Global Existence}
	\begin{definition}\rm
		Let $T>0$ and $U_0\in\mathcal H_s$.
		
		\begin{itemize}
			\item A function
			$
			U\in C([0,T),\mathcal H_s)
			$
			is called a \textbf{mild solution} of problem \eqref{abstrac-cauchy} on $[0,T)$ if it satisfies the integral equation
			\begin{equation}\label{integral-equation}
				U(t)
				=
				e^{At}U_0
				+
				\int_0^t
				e^{A(t-s)}B(U(s))\,ds,
				\qquad
				t\in[0,T).
			\end{equation}
			
			\item A function
			$
			U:[0,T)\rightarrow\mathcal H_s
			$
			is called a \textbf{strong solution} (regular) of problem \eqref{abstrac-cauchy} on $[0,T)$, if $U$ is continuous on $[0,T)$, continuously differentiable on $(0,T)$, with $U\in D(A)$ for $t\in(0,T)$, and \eqref{abstrac-cauchy} is satisfied on $[0,T)$ almost everywhere.
			
		\end{itemize}
	\end{definition}
	
	We are now ready to state the main existence result for the abstract Cauchy problem \eqref{abstrac-cauchy}, which is equivalent to problem \eqref{P}.
	\begin{theorem} 
		\label{theo-existence} Under the {Assumption \ref{Assumption}} with $h\in L^2(\Omega)$, it holds the following statements:
		\begin{itemize}
			\item[{\bf(I)}]   If $U_0\in \mathcal{H}_s$, then there exists
			$T_{\max}>0$  such that problem \eqref{abstrac-cauchy} has a unique mild
			solution $U\in C([0,T_{\max}),\mathcal{H}_s)$.
			\item[{\bf(II)}]   If  $U_0\in D({A})$, then the above mild solution $U$ is regular one.
			\item[{\bf(III)}] In both cases, we have that $T_{\max}=+\infty$.
		\end{itemize}
	\end{theorem}
	\begin{proof}
		The proof of Theorem \ref{theo-existence} follows from a standard application of the semigroup theory for linear operators developed by Pazy \cite{Pazy}. \textit{Step 1.} We show that the operator
		$
		A:D(A)\subset \mathcal{H}_s\rightarrow \mathcal{H}_s,
		$
		defined in \eqref{def_A}, is the infinitesimal generator of a contraction $C_0$-semigroup on $\mathcal{H}_s$. This is achieved by proving that $A$ is dissipative and maximal dissipative, and then applying the Lumer--Phillips Theorem (\cite[Theorem 1.4.3]{Pazy}). \textit{Step 2.} We prove that the nonlinear operator
		$
		B:\mathcal{H}_s\rightarrow \mathcal{H}_s
		$
		is locally Lipschitz continuous. Combining \textit{Steps 1} and \textit{2}, the local existence of mild solutions follows from \cite[Theorem 6.1.4]{Pazy}, while the local existence of regular solutions is obtained by applying \cite[Theorem 6.1.5]{Pazy}. \textit{Step 3.} Finally, we derive suitable \emph{a priori} estimates to extend the local solutions globally in time, thereby establishing the existence of global mild and regular solutions.
		
		For the reader's convenience, we provide the details of each step below.
		
		\paragraph*{\it Step 1.} The operator $A$ defined in (\ref{def_A}) is the infinitesimal generator of a $C_0$-semigroup in $\mathcal{H}$. Indeed, we take arbitrary element $U\in D(A)$.  We oberve that
		\begin{align*}(v,u)_{H^s_0(\Omega)}&=\int_{\Omega}\int_{\Omega} \frac{v(x)-v(y)}{|x-y|^{n+2s}}(u(x)-u(y))dydx\\ 
			&=\int_{\Omega}\int_{\Omega} \frac{u(x)-u(y)}{|x-y|^{n+2s}}dy v(x)dx-\int_{\Omega}\int_{\Omega} \frac{u(x)-u(y)}{|x-y|^{n+2s}}v(y)dy dx\\  
			&=\int_{\Omega}2\int_{\Omega} \frac{u(x)-u(y)}{|x-y|^{n+2s}}dy v(x)dx=((-\Delta)^su,v).
		\end{align*}
		Then
		\begin{eqnarray*}
			\left<AU,U\right>_{\mathcal{H}_s}=(v,u)_{H^s_0}+(-(-\Delta)^su,v)=0,
		\end{eqnarray*}
		which proves dissipativity for $A$. To show that $A$ is maximal we need to prove that $R(I-A)=\mathcal{H}$, where $R(I-A)$ is the range of $I-A$. Indeed, let $U^*=(u^*,v^*)\in \mathcal{H}_s$, and consider the equation $(I-A)U=U^*$ which, written in components, reads
		\begin{equation}\label{max_A}
			\left\{
			\begin{array}{rcl}
				u-v &=& u^* \in H_0^s(\Omega),\\[0.3em]
				v+(-\Delta)^s u &=& v^* \in L^2(\Omega).
			\end{array}
			\right.
		\end{equation}
		Substituting $v=u-u^*$ in the second equations of (\ref{max_A}), we obtain
		\begin{align*}
			u+(-\Delta)^s u=v^*+u^*=:w^*\in L^2(\Omega).
		\end{align*}
		Since the corresponding weak formulation is
		$$a(u,w)=\int_{\Omega}w^*wdx,\quad \text{for all } w\in H^s_0(\Omega),$$
		where
		$$a(u,w):=\int_{\Omega}\left[\,uw+(-\Delta)^s u w\,\right]dx,$$
		the bilinear functional $a$ is continuous and coercive, by the Lax-Milgram Theorem we can find a weak solution $u\in H^s_0(\Omega)$ and we get $v=u-u^* \in H^s_0(\Omega)$, and therefor $(-\Delta)^s u=v^*-v\in L^2(\Omega)$. Hence, $R(I-A)=\mathcal{H}_s$, which shows that $A$ is maximal monotone. Consequently, by the Lumer--Phillips Theorem, $A$ generates a contraction $C_0$-semigroup on $\mathcal{H}$.
		
		\paragraph{\it Step 2.}
		The operator $B:\mathcal{H}_s\to\mathcal{H}_s$ defined in
		\eqref{def_A} is locally Lipschitz continuous. Indeed, let $R>0$ and
		let $U^1=(u^1,u_t^1),\,U^2=(u^2,u_t^2)\in\mathcal{H}_s$ satisfy
		\[
		\|U^1\|_{\mathcal{H}_s},\,
		\|U^2\|_{\mathcal{H}_s}\le R.
		\]
		Setting $w=u^1-u^2$, we have
		\begin{align}\label{diff_operator_B}
			\|B(U^1)-B(U^2)\|_{\mathcal H_s}
			&=
			\left\|
			\gamma\left(
			\|U^1\|_{\mathcal H_s}^{2q}u_t^1
			-
			\|U^2\|_{\mathcal H_s}^{2q}u_t^2
			\right)
			-
			\left(f(u^1)-f(u^2)\right)
			\right\|.
		\end{align}
		We estimate each term on the right-hand side separately. Using the identity
		$$
		\|U^1\|_{\mathcal H_s}^{2q}u_t^1
		-
		\|U^2\|_{\mathcal H_s}^{2q}u_t^2
		=
		\|U^1\|_{\mathcal H_s}^{2q}(u_t^1-u_t^2)
		+
		\bigl(
		\|U^1\|_{\mathcal H_s}^{2q}
		-
		\|U^2\|_{\mathcal H_s}^{2q}
		\bigr)u_t^2,
		$$
		we obtain
		$$
		\left\|
		\gamma\|U^1\|_{\mathcal H_s}^{2q}(u_t^1-u_t^2)
		\right\|
		\le
		\gamma R^{2q}\|u_t^1-u_t^2\|
		\le
		\gamma R^{2q}\|U^1-U^2\|_{\mathcal H_s}.
		$$
		Moreover, by the Mean Value Theorem, there exists
		$\theta\in(0,1)$ such that
		\begin{align*}
			\gamma
			\Bigl|
			\|U^1\|_{\mathcal H_s}^{2q}
			-
			\|U^2\|_{\mathcal H_s}^{2q}
			\Bigr|
			&=
			2\gamma q
			\Bigl(
			\theta\|U^1\|_{\mathcal H_s}
			+(1-\theta)\|U^2\|_{\mathcal H_s}
			\Bigr)^{2q-1}
			\Bigl|
			\|U^1\|_{\mathcal H_s}
			-
			\|U^2\|_{\mathcal H_s}
			\Bigr|\\
			&\le
			2^{2q}\gamma q
			R^{2q-1}
			\|U^1-U^2\|_{\mathcal H_s}.
		\end{align*}
		Finally, by the Mean Value Theorem, Assumption \eqref{assumption_f'}, H\"older's inequality with
		$
		\frac{2p}{2(p+1)}
		+\frac{2}{2(p+1)}=1,
		$
		and the continuous embedding
		\(H_0^s(\Omega)\hookrightarrow L^{2(p+1)}(\Omega)\), we obtain
		\begin{eqnarray*}
			\left\|f(u^1)-f(u^2)\right\|
			&\le&
			C
			\left(1+\|u^1\|^{p}_{2(p+1)}+\|u^2\|^{p}_{2(p+1)}\right)
			\|u^1-u^2\|_{2(p+1)}\\
			&\le & C
			\left(1+\|u^1\|^{p}_{H^s_0(\Omega)}+\|u^2\|^{p}_{H^s_0(\Omega)}\right)
			\|u^1-u^2\|_{H^s_0(\Omega)}
			\\
			&\le&
			C_R
			\|U^1-U^2\|_{\mathcal H_s}.
		\end{eqnarray*}
		Thus, returning to \eqref{diff_operator_B}, there exists a constant $C_R>0$ such that
		\begin{align}\label{Lipschitz-cond}
			||B(U^1)-B(U^2)||_{\mathcal{H}_s} \le C_R||U^1-U^2||_{\mathcal{H}_s}.
		\end{align}
		This proves that the operator $B$ is locally Lipschitz.
		
		Since we have already proved in \textit{Step 1} that the operator
		$
		A:D(A)\subset\mathcal H_s\to\mathcal H_s
		$
		is the infinitesimal generator of a $C_0$-semigroup
		$\{e^{tA}\}_{t\ge0}$ on $\mathcal H_s$, and in \textit{Step 2} that the
		nonlinear operator
		$
		F:\mathcal H_s\to\mathcal H_s
		$
		is locally Lipschitz continuous, all the assumptions of the abstract
		existence theorem for semilinear evolution equations are satisfied. Therefore, by applying \cite[Chapter 6, Theorem 1.4]{Pazy}, for every
		initial datum $U_0\in\mathcal H_s$, the Cauchy problem
		\eqref{abstrac-cauchy} admits a unique mild solution
		$
		U\in C([0,T_{\max});\mathcal H_s),
		$
		for some maximal existence time $T_{\max}>0$, satisfying the
		variation-of-constants formula \eqref{integral-equation}. Furthermore, if $U_0\in D(A)$, then by
		\cite[Chapter 6, Theorem 1.5]{Pazy} the mild solution is regular and
		satisfies
		$$
		U\in C([0,T_{\max});D(A))
		\cap C^1([0,T_{\max});\mathcal H_s).
		$$
		This completes the proof of statements \textbf{(I)} and \textbf{(II)} of
		Theorem~\ref{theo-existence}.

		\paragraph{\it Step 3.}
		It remains to check that both mild and regular solutions are globally defined, that is,  $T_{\max}=+\infty.$ Indeed, multiplying the problem (\ref{P}) by $u_t$ and integrating over $\Omega$ we get
		\begin{align}\label{unif-1}
			\int_{\Omega}u_{tt}u_tdx+\int_{\Omega}(-\Delta)^{s}uu_tdx+\gamma\|U\|^{2q}_{\mathcal{H}_s}\int_{\Omega} u_tu_t dx+\int_{\Omega}f(u)u_tdx=\int_{\Omega}hu_tdx.
		\end{align}
		Through direct calculations, it follows from \eqref{unif-1} that
		\begin{equation}\label{der-energy}
			\frac{d}{dt}{E}(t)+\gamma\|U\|^{2q}_{\mathcal{H}_s}\|u_t\|^2=0,
		\end{equation}
		where we remember that $E$ is defined in \eqref{energy-formula}. Equation \eqref{der-energy}
		implies 
		\begin{align}\label{global_0}
			{E}(t)+\gamma\int_0^t\|U\|^{2q}_{\mathcal{H}_s}\|u_t\|^2ds={E}(0),\quad \forall t\in [0,T_{\max}).
		\end{align}
		Now, from assumption (\ref{assumption_f}) and using Poincar\'e inequality (\ref{poincare}), we have
		\begin{align}\label{global_a}
			\int_{\Omega}F(u)dx\ge -\frac{c_f}{2}\|u\|^2-C_f|\Omega|\ge -\frac{c_fK^2}{2}\|u\|^2_{H^s_0(\Omega)}-C_f|\Omega|.
		\end{align}
		On the other hand, from H\"older and Young inequalities, and using again (\ref{poincare}), we have
		\begin{align}\label{global_b}
			\int_{\Omega}hudx\le \|h\|\|u\|\le \|h\|K\|u\|_{H^s_0(\Omega)}\le \frac{K^2}{\kappa}\|h\|^2+\frac{\kappa}{4}\|u\|^2_{H^s_0(\Omega)}.
		\end{align}
		Then, from definition of $E(t)$ and inequalities (\ref{global_a}) and (\ref{global_b}), we obtain that
		\begin{eqnarray}\label{global-in-time1}
			E(t) \geq \frac{1}{2}\|u_t\|^2+\frac{\kappa}{4}\|u\|^2_{H^s_0(\Omega)}-C_f|\Omega|-\frac{K^2}{\kappa}\|h\|^2\ge \frac{\kappa}{4}||U(t)||^2_{\mathcal{H}_s}-\kappa_0,
		\end{eqnarray}
		where $\kappa_0=\left[C_f|\Omega|+\frac{K^2}{\kappa}\|h\|^2\right]$. 
		
		Hence, from (\ref{global-in-time1}), (\ref{global_0}), and using that $\|U\|^{2q}_{\mathcal{H}_s}\|u_t\|^2\ge \|u_t\|^{2(q+1)}$, we obtain
		\begin{eqnarray}\label{global-in-time}
			||U(t)||^2_{\mathcal{H}_s}+\int_0^t\|u_t\|^{2(q+1)}ds\le \frac{E(0)+\kappa_0}{\min\{\kappa/4,\gamma\}},\quad \forall t\in [0,T_{\max}).
		\end{eqnarray}
		Estimate \eqref{global-in-time} implies that every (mild or strong) solution remains uniformly bounded in time. Therefore, by \cite[Theorem 6.1.4]{Pazy}, we conclude that $T_{\max}=+\infty$. This proves item \textbf{(III)} and completes the proof of Theorem~\ref{theo-existence}.
	\end{proof}
	
	\subsection{Continuous Dependence}
	
	\begin{theorem}\label{theo-cont-dependence}
		Let $U^1(t)$ and $U^2(t)$ be two mild (or strong) solutions of problem \eqref{P} corresponding to the initial data $U_0^1$ and $U_0^2$, respectively. Then
		\begin{equation}\label{continuous-dependence}
			\|U^1(t)-U^2(t)\|_{\mathcal H_s}
			\le
			e^{Ct}
			\|U_0^1-U_0^2\|_{\mathcal H_s},
			\qquad
			\forall\, t\in[0,T),
		\end{equation}
		where
		$
		C=C\bigl(\|U_0^1\|_{\mathcal H_s},\|U_0^2\|_{\mathcal H_s}\bigr)>0.
		$
	\end{theorem}
	
	\begin{proof}
		Let
		$
		W(t)=U^1(t)-U^2(t),$ $W_0=U_0^1-U_0^2.$
		Since $U^1$ and $U^2$ are mild solutions, the variation-of-constants formula \eqref{integral-equation} yields
		\begin{equation*}
			W(t)
			=
			e^{At}W_0
			+
			\int_0^t
			e^{A(t-s)}
			\left(
			B(U^1(s))-B(U^2(s))
			\right)\,ds,
			\qquad
			t\in[0,T).
		\end{equation*}
		Since $\{e^{At}\}_{t\ge0}$ is a contraction semigroup on $\mathcal H_s$, it follows from \eqref{Lipschitz-cond} that
		\begin{align*}
			\|W(t)\|_{\mathcal H_s}
			&\le
			\|W_0\|_{\mathcal H_s}
			+
			\int_0^t
			\|B(U^1(s))-B(U^2(s))\|_{\mathcal H_s}\,ds
			\\
			&\le
			\|W_0\|_{\mathcal H_s}
			+
			C
			\int_0^t
			\|W(s)\|_{\mathcal H_s}\,ds,
		\end{align*}
		where $
		C=C\bigl(\|U_0^1\|_{\mathcal H_s},\|U_0^2\|_{\mathcal H_s}\bigr)>0.$
		An application of Gronwall's inequality therefore yields
		\begin{equation*}
			\|W(t)\|_{\mathcal H_s}
			\le
			e^{Ct}
			\|W_0\|_{\mathcal H_s},
			\qquad
			t\in[0,T),
		\end{equation*}
		which is exactly \eqref{continuous-dependence}. The proof is complete.
	\end{proof}
	
	\section{Global attractor}
	\subsection{Preliminary Abstract Results}\setcounter{equation}{0}
	For the sake of completeness, we recall some abstract results on dynamical systems following the books by Chueshov and Lasiecka \cite{chueshov-white,chueshov-yellow} and Chueshov \cite{chueshov-2015}.
	
	\medskip
	We briefly recall the notion of a dynamical system that will be used throughout the paper. A \emph{dynamical system} is a pair $(X,S_t)$, where $X$ is a complete metric space and $\{S_t\}_{t\ge0}$ is a family of continuous mappings from $X$ into itself satisfying
	$$
	S_0=I,\qquad
	S_{(t+s)}=S_t\text{o}S_s, \qquad t,s\ge0,
	$$
	and such that, for every $z\in X$, the map
	$
	t\mapsto S_tz
	$
	is continuous from $[0,\infty)$ into $X$. The family $\{S_t\}_{t\ge0}$ is called the \emph{evolution semigroup} associated with the dynamical system.
	
	\medskip
	Let $(X,S_t)$ be a dynamical system given by a $C_0$-semigroup $S_t$ on a complete metric space $X$. Then a global attractor for $S_t$ is a compact set $\mathcal{A}\subset X$ that is fully invariant and uniformly attracting, that is, $S_t\mathcal{A}=\mathcal{A}$ for all $t\ge 0$ and for every bounded subset $B\subset X$
	$$
	\mbox{dist}(S_tB,\mathcal{A}) = \sup_{x\in S_tB}\inf_{y\in
		\mathcal{A}}d(x,y) \to 0 \quad \mbox{as} \quad t \to \infty.
	$$

	
	\medskip
	A dynamical system $(X,S_t)$ is \textit{asymptotically smooth} in $X$ if for any bounded positive invariant set $B \subset X$, there exists a compact set $K \subset \overline{B}$, such that
	$$
	\mbox{dist}(S_tB,K) = 0 \quad \mbox{as} \quad t \to \infty .
	$$
	
	\medskip
	Let $\mathcal{N}$ be the set of \textit{stationary points} of a dynamical system $(X,S_t)$:
	$$\mathcal{N}=\{v\in X:S_tv=v\;\;\mbox{for all}\;\;t\ge 0\}.$$
	We define the {\it unstable manifold} $\mathrm{M}^u(\mathcal{N})$ emanating from set $\mathcal{N}$ as a set of all $y\in X$ such that there exists a full trajectory $\Upsilon=\{u(t):t\in \mathbb{R}\}$ with the properties
	$$u(0)=y\quad \mbox{and}\quad \lim_{t\rightarrow-\infty}\mbox{dist}_{X}(u(t),\mathcal{N})=0.$$
	
	\medskip
	Let $Y\subseteq X$ be a forward invariant set of a dynamical system
	$(X,S_t)$.
	\begin{itemize}
		\item A continuous functional $\Phi:Y\to\mathbb{R}$ is said to be a
		\emph{Lyapunov function} for the dynamical system $(X,S_t)$ on $Y$ if and only if,
		for every $y\in Y$, the function
		$
		t\mapsto \Phi(S_ty)
		$
		is non-increasing.
		\item The Lyapunov function $\Phi$ is said to be \emph{strict} on $Y$ if
		$
		\Phi(S_ty)=\Phi(y),$ $\forall\, t>0,
		$
		for some $y\in Y$, implies that
		$
		S_ty=y,$ $\forall\, t>0,
		$
		that is, $y$ is a stationary point of $(X,S_t)$.
		
		\item The dynamical system $(X,S_t)$ is said to be \emph{gradient} if and only if
		there exists a strict Lyapunov function for $(X,S_t)$ on the whole phase space
		$X$.
	\end{itemize}

	\medskip
	The following theorem established by Chueschov and Lasiecka in \cite{chueshov-yellow} gives the conditions for the existence of a compact global attractor for gradient and asymptotically smooth systems.
	\begin{theorem}[Corollary 7.5.7, \cite{chueshov-yellow}] \label{theo-chueshov}
		Assume that $(X,S_t)$ is a gradient asymptotically smooth dynamical system with Lyapunov functional denoted by $\Phi$. Assume in addition that,
		\begin{itemize}
			\item[{\bf I.}] $\Phi(U)$ is a bounded from above on any bounded subset of $X$,
			\item[{\bf II.}] the set $\Phi_R=\{U|\Phi(U)\le R\}\}$ is bounded for every $R$,
			\item[{\bf III.}] the set of stationary points $\mathcal{N}$ is bounded.
		\end{itemize}
		Then the system $(X,S_t)$ has a compact global attractor characterized by $\mathfrak{A}=\mathrm{M}^u(\mathcal{N})$.
	\end{theorem}
	
	\subsection{Main result}
	
	As a consequence of Theorems~\ref{theo-existence} and~\ref{theo-cont-dependence}, problem~\eqref{P} generates a nonlinear dynamical system $(\mathcal H_s,S_t)$, where
	\begin{equation*}
		S_t:\mathcal H_s\to\mathcal H_s,
		\qquad
		S_tU_0=U(t),
		\quad
		t\ge0,
	\end{equation*}
	and $U(t)$ denotes the unique global mild solution of problem~\eqref{P} corresponding to the initial datum $U_0\in\mathcal H_s$. Moreover, the solution semigroup $\{S_t\}_{t\ge0}$ is Lipschitz continuous with respect to the initial data.
	
	The remainder of the paper is devoted to the study of the long-time behavior of the dynamical system $(\mathcal H_s,S_t)$. In particular, we establish the existence of a compact global attractor.
	
	The main result of this paper is stated in the following theorem.
	\begin{theorem}{\bf[Global attractor]}\label{theo-main}
		Assume that the hypotheses of Theorem~\ref{theo-existence} are satisfied with $q\in[1/2,2]$. Then the dynamical system $(\mathcal H_s,S_t)$ generated by problem~\eqref{P} possesses a compact global attractor
		$$
		\mathfrak{A}=\mathrm{M}^u(\mathcal N)
		$$
		in the phase space
		$
		\mathcal H_s=H_0^s(\Omega)\times L^2(\Omega).
		$
	\end{theorem}
	\begin{proof}
		The proof of Theorem~\ref{theo-main} is based on the abstract criterion for the existence of global attractors given in Theorem~\ref{theo-chueshov}. By Theorem~\ref{theo-existence}, problem~\eqref{P} generates a continuous dynamical system $(\mathcal H_s,S_t)$. Proposition~\ref{prop-gradient-system} shows that the dynamical system is gradient by proving that the energy functional
		$
		\Phi:=E
		$
		is a strict Lyapunov function, while Proposition~\ref{prop-asymptotic} establishes its asymptotic smoothness. Furthermore, Proposition~\ref{Prop-stationary-set} proves that the Lyapunov function is bounded from above on bounded subsets of $\mathcal H_s$ and that all its sublevel sets are bounded. Finally, Proposition~\ref{Prop-stationary-set-2} shows that the set $\mathcal N$ of stationary points is bounded in $\mathcal H_s$. Therefore, all the assumptions of Theorem~\ref{theo-chueshov} are satisfied, and the existence of a compact global attractor
		$
		\mathfrak A=\mathrm{M}^u(\mathcal N)
		$
		follows immediately.
	\end{proof}
	In the following, we devote our attention to the proofs of Propositions~\ref{prop-gradient-system}, \ref{prop-asymptotic}, \ref{Prop-stationary-set}, and~\ref{Prop-stationary-set-2}.
	
	\begin{proposition}\label{prop-gradient-system}
		The dynamical system $(\mathcal H_s,S_t)$ generated by problem~\eqref{P} is gradient.
	\end{proposition}
	
	\begin{proof}
		Let $\Phi:=E$ be the energy functional defined in \eqref{energy-formula}. Given
		$V=(u_0,u_1)\in\mathcal H_s$, let
		$U(t)=S_tV=(u(t),u_t(t))$ be the corresponding solution of
		problem~\eqref{P}.
		
		Multiplying \eqref{P} by $u_t$ and integrating over
		$\Omega\times(0,t)$, we obtain
		\begin{equation}\label{grad-1}
			\Phi(U(t))
			+\int_0^t
			\|U(s)\|_{\mathcal H_s}^{2q}
			\|u_t(s)\|^2\,ds
			=
			\Phi(V),
			\qquad t\ge0.
		\end{equation}
		Hence, $\Phi(S_tV)$ is non-increasing.
		
		Now suppose that
		$$
		\Phi(S_tV)=\Phi(V),
		\qquad
		\forall\,t\ge0.
		$$
		Then, by \eqref{grad-1},
		\begin{align*}
			\int_0^t
			\|U(s)\|_{\mathcal H_s}^{2q}
			\|u_t(s)\|^2\,ds
			=0,
			\qquad
			\forall\,t\ge0.
		\end{align*}
		Since
		$$
		\|U(s)\|_{\mathcal H_s}^{2q}
		\ge
		\|u_t(s)\|^{2q},
		$$
		we obtain
		$$
		\|U(s)\|_{\mathcal H_s}^{2q}
		\|u_t(s)\|^2
		\ge
		\|u_t(s)\|^{2(q+1)}.
		$$
		Therefore,
		$$
		u_t(t)=0
		\quad\text{in }L^2(\Omega)\text{ for a.e. }t\ge0.
		$$
		Since
		\[
		u_t\in C([0,\infty);L^2(\Omega)),
		\]
		it follows that
		\[
		u_t(t)\equiv0
		\quad\text{in }L^2(\Omega),\qquad \forall\,t\ge0.
		\]
		Hence, $u$ is time-independent, that is,
		\[
		S_tV=V,
		\qquad \forall\,t\ge0.
		\]
		Therefore, $\Phi$ is a strict Lyapunov function, and the dynamical system
		$(\mathcal H_s,S_t)$ is gradient.
	\end{proof}
	
	To prove Proposition~\ref{prop-asymptotic}, we first establish the following two auxiliary results, which will be used in the proof of Lemma~\ref{lemma-stability}. The first provides an inequality that will be used to control the nonlocal terms according to the size of the exponent $q$, while the second gives a convenient criterion for verifying the asymptotic smoothness of the dynamical system. The stability estimate obtained in Lemma~\ref{lemma-stability} will then be used to prove Proposition~\ref{prop-asymptotic}.
	\begin{lemma}\cite[Lemma 4.5]{Zhou-Yang-2026}\label{lemma-Zhou_Yang}
		For any $q>0$.
		\begin{align*}
			0<\frac{a^q-b^q}{a-b}\le \max\{1,q\}\left(\frac{a^q+b^q}{a+b}\right),
		\end{align*}
		for all $a, b\geq 0$ with $a+b>0$ and $a\neq b$.
	\end{lemma}
	\begin{theorem} [Theorem 7.1.11, \cite{chueshov-yellow}] \label{theo-asymp-smootness} Let $(X,S_t)$ be a dynamical system on a complete metric space $X$ endowed with a metric $d$. Assume that for any bounded positively invariant set $B$ in $X$ and for any $\epsilon> 0$ there exists $T=T_{\epsilon,B}$ such that
		\begin{align*}
			d(S_Ty_1,S_Ty_2)
			\le
			\epsilon+\Psi_{\epsilon,B,T}(y_1,y_2),
			\qquad y_i\in B.
		\end{align*}
		where $\Psi_{\epsilon,B,T}(y_1,y_2)$ is a functional defined on $B\times B$ such that
		\begin{align*}\liminf_{m\rightarrow\infty}\liminf_{n\rightarrow \infty}\Psi_{\epsilon,B,T}(y_n,y_m)=0\end{align*}
		for every sequence ${y_n}$ from $B.$ Then $(X,S_t)$ is an asymptotically smooth dynamical system.
	\end{theorem}
	
	\begin{lemma}\label{lemma-stability}
		Assume that the hypotheses of Theorem~2.1 are satisfied with
		$q\in[1/2,2]$. Let
		$
		U^i(t)=S(t)U_0^i=(u^i(t),u_t^i(t)),$ $i=1,2,$
		be two mild solutions of problem~\eqref{P} corresponding to the initial
		data $U_0^1,U_0^2\in B$, where $B$ is a bounded subset of
		$\mathcal{H}_s$. Define
		\begin{align*}
			w(t)=u^1(t)-u^2(t).
		\end{align*}
		Under the assumptions of Theorem~\ref{theo-main}, there exist positive
		constants $C_B$ and $C_{B,T}$ such that
		\begin{eqnarray}\label{inequality-main}
			\|S_TU_0^1-S_TU_0^2\|_{\mathcal H_s}^2
			&\le&
			\frac{C_B}{T^{1/q+1}}
			+
			C_{B,T}
			\sup_{s\in[0,T]}
			\|w(s)\|_{p+2}
			\nonumber\\
			&&+
			\frac{2}{T}
			\left|
			\int_{0}^{T}
			\int_{\tau}^{T}
			\int_{\Omega}
			\bigl(f(u^{1})-f(u^{2})\bigr)
			w_t\,dx\,ds\,d\tau
			\right|.
		\end{eqnarray}
	\end{lemma}
	\begin{proof}
		Let $W=U^1-U^2=(w,w_t).
		$
		For the moment, assume that
		$
		U^1(t)\neq U^2(t),$ $t\ge0.
		$
		Then \(W\) satisfies
		\begin{equation}\label{difference-equation}
			\left\{
			\begin{aligned}
				&w_{tt}+(-\Delta)^sw
				+\gamma\left(
				\|U^1\|_{\mathcal H_s}^{2q}u_t^1
				-
				\|U^2\|_{\mathcal H_s}^{2q}u_t^2
				\right)
				+\bigl(f(u^1)-f(u^2)\bigr)=0,
				\\
				&w(0)=u_0^1-u_0^2,
				\qquad
				w_t(0)=u_1^1-u_1^2.
			\end{aligned}
			\right.
		\end{equation}
		The nonlinear damping term admits the decomposition
		\begin{align*}
			\|U^1\|_{\mathcal H_s}^{2q}u_t^1
			-
			\|U^2\|_{\mathcal H_s}^{2q}u_t^2
			&=
			\frac12
			\left(
			\|U^1\|_{\mathcal H_s}^{2q}
			+
			\|U^2\|_{\mathcal H_s}^{2q}
			\right)w_t+
			\frac12
			\left(
			\|U^1\|_{\mathcal H_s}^{2q}
			-
			\|U^2\|_{\mathcal H_s}^{2q}
			\right)
			(u_t^1+u_t^2).
		\end{align*}
		For simplicity of notation, we introduce
		\[
		\Pi_j(t):=
		\|U^1(t)\|_{\mathcal H_s}^{2q}
		+(-1)^{j-1}
		\|U^2(t)\|_{\mathcal H_s}^{2q},
		\qquad j=1,2.
		\]
		Hence, multiplying \eqref{difference-equation} by $w_t$, integrating over
		$\Omega\times(\tau,T)$, and using the above notation, we obtain
		\begin{eqnarray}\label{asyp-1}
			\begin{aligned}
				&\frac{1}{2}\|W(T)\|_{\mathcal H_s}^{2}
				-\frac{1}{2}\|W(\tau)\|_{\mathcal H_s}^{2}
				+\frac{\gamma}{2}
				\int_{\tau}^{T}
				\Pi_1(s)
				\|w_t(s)\|^2\,ds
				\\
				&
				+\underbrace{
					\frac{\gamma}{2}
					\int_{\tau}^{T}
					\Pi_2(s)
					\int_{\Omega}
					(u_t^1+u_t^2)w_t\,dx\,ds
				}_{I}=
				-\int_{\tau}^{T}
				\int_{\Omega}
				(f(u^1)-f(u^2))w_t\,dx\,ds .
			\end{aligned}
		\end{eqnarray}
		Observe that whenever
		$
		\|U^1(t)\|_{\mathcal H_s}
		=
		\|U^2(t)\|_{\mathcal H_s},
		$
		the quantity $I$ vanishes identically. Hence, it suffices to consider the case
		$
		\|U^1(t)\|_{\mathcal H_s}
		\neq
		\|U^2(t)\|_{\mathcal H_s},
		$
		throughout the remainder of the proof. Accordingly, we define
		$$
		\Theta(t):=
		\frac{
			\|U^1(t)\|_{\mathcal H_s}^{2q}
			-
			\|U^2(t)\|_{\mathcal H_s}^{2q}}
		{
			\|U^1(t)\|_{\mathcal H_s}^{2}
			-
			\|U^2(t)\|_{\mathcal H_s}^{2}}.
		$$
		Consequently, the term $I$ can be expressed as
		\begin{align*}
			I
			&=
			\frac{\gamma}{2}
			\int_{\tau}^{T}
			\Theta(s)
			\left(
			\|U^{1}(s)\|_{\mathcal H_s}^{2}
			-
			\|U^{2}(s)\|_{\mathcal H_s}^{2}
			\right)
			\left(
			\|u_t^{1}(s)\|^{2}
			-
			\|u_t^{2}(s)\|^{2}
			\right)\,ds
			\\
			&=
			\frac{\gamma}{2}
			\int_{\tau}^{T}
			\Theta(s)
			\left(
			\|u_t^{1}(s)\|^{2}
			-
			\|u_t^{2}(s)\|^{2}
			\right)^2ds
			\\
			&\quad
			+
			\underbrace{\frac{\gamma}{2}
				\int_{\tau}^{T}
				\Theta(s)
				\left(
				\|u^{1}(s)\|_{H_0^s(\Omega)}^{2}
				-
				\|u^{2}(s)\|_{H_0^s(\Omega)}^{2}
				\right)
				\left(
				\|u_t^{1}(s)\|^{2}
				-
				\|u_t^{2}(s)\|^{2}
				\right)\,ds}_{I_1}.
		\end{align*}
		Substituting this identity into \eqref{asyp-1}, we obtain
		\begin{align}\label{asymp-2}
			&\frac12\|W(T)\|_{\mathcal H_s}^{2}
			-\frac12\|W(\tau)\|_{\mathcal H_s}^{2}
			+\frac{\gamma}{2}
			\int_{\tau}^{T}
			\Pi_1(s)
			\|w_t(s)\|^{2}\,ds
			\nonumber\\
			&\quad
			+\frac{\gamma}{2}
			\int_{\tau}^{T}
			\Theta(s)
			\left(
			\|u_t^{1}(s)\|^{2}
			-
			\|u_t^{2}(s)\|^{2}
			\right)^2\,ds=
			-I_1
			-
			\int_{\tau}^{T}
			\int_{\Omega}
			\bigl(f(u^{1})-f(u^{2})\bigr)
			w_t\,dx\,ds.
		\end{align}
		By Lemma~\ref{lemma-Zhou_Yang}, we have
		$$
		0<\Theta(s)
		\le
		\max\{1,q\}
		\frac{
			\|U^{1}(s)\|_{\mathcal H_s}^{2q}
			+
			\|U^{2}(s)\|_{\mathcal H_s}^{2q}
		}{
			\|U^{1}(s)\|_{\mathcal H_s}^{2}
			+
			\|U^{2}(s)\|_{\mathcal H_s}^{2}
		}.
		$$
		Moreover, applying the Cauchy--Schwarz and Young inequalities, we obtain
		\begin{align*}
			-I_1
			&\leq
			\frac{\gamma}{2}
			\int_{\tau}^{T}
			\Theta(s)
			\left(
			\|u_t^{1}(s)\|^{2}
			-
			\|u_t^{2}(s)\|^{2}
			\right)^2\,ds
			\\
			&\quad
			+
			\frac{\gamma}{8}
			\int_{\tau}^{T}
			\Theta(s)
			\left(
			\|u^{1}(s)\|_{H_0^s(\Omega)}^{2}
			-
			\|u^{2}(s)\|_{H_0^s(\Omega)}^{2}
			\right)^2\,ds .
		\end{align*}
		To estimate the second term on the right-hand side, we first observe that
		\begin{align*}
			&\|u^{1}(s)\|_{H_0^s(\Omega)}^{2}
			-
			\|u^{2}(s)\|_{H_0^s(\Omega)}^{2}
			\\
			&=
			\left(
			\|u^{1}(s)\|_{H_0^s(\Omega)}
			+
			\|u^{2}(s)\|_{H_0^s(\Omega)}
			\right)
			\left(
			\|u^{1}(s)\|_{H_0^s(\Omega)}
			-
			\|u^{2}(s)\|_{H_0^s(\Omega)}
			\right).
		\end{align*}
		Consequently
		\begin{align*}
			\left|
			\|u^{1}(s)\|_{H_0^s(\Omega)}^{2}
			-
			\|u^{2}(s)\|_{H_0^s(\Omega)}^{2}
			\right|^2
			\leq
			\left(
			\|U^{1}(s)\|_{\mathcal H_s}^{2}
			+
			\|U^{2}(s)\|_{\mathcal H_s}^{2}
			\right)
			\|w(s)\|_{H_0^s(\Omega)}^{2}.
		\end{align*}
		Substituting the above estimate into the previous inequality yields
		\begin{align*}
			-I_1
			&\leq
			\frac{\gamma}{2}
			\int_{\tau}^{T}
			\Theta(s)
			\left(
			\|u_t^{1}(s)\|^{2}
			-
			\|u_t^{2}(s)\|^{2}
			\right)^2\,ds
			\\
			&\quad
			+
			\frac{\gamma}{8}
			\int_{\tau}^{T}
			\Theta(s)
			\left(
			\|U^{1}(s)\|_{\mathcal H_s}^{2}
			+
			\|U^{2}(s)\|_{\mathcal H_s}^{2}
			\right)
			\|w(s)\|_{H_0^s(\Omega)}^{2}\,ds .
		\end{align*}
		Thus, using the estimate for $\Theta$ established above, we conclude that
		\begin{eqnarray*}
			-I_1
			\leq
			\frac{\gamma}{2}
			\int_{\tau}^{T}
			\Theta(s)
			\left(
			\|u_t^{1}(s)\|^{2}
			-
			\|u_t^{2}(s)\|^{2}
			\right)^2\,ds
			+
			\frac{\gamma\max\{1,q\}}{8}
			\int_{\tau}^{T}
			\Pi_1(s)
			\|w(s)\|_{H_0^s(\Omega)}^{2}\,ds .
		\end{eqnarray*}
		Returning to \eqref{asymp-2}, we get
		\begin{align}\label{difference-energy-2}
			&\frac12\|W(T)\|_{\mathcal H_s}^{2}
			-\frac12\|W(\tau)\|_{\mathcal H_s}^{2}
			+\frac{\gamma}{2}
			\int_{\tau}^{T}
			\Pi_1(s)
			\|w_t(s)\|^{2}\,ds
			\nonumber\\
			&\le
			\frac{\gamma\max\{1,q\}}{4}
			\int_{\tau}^{T}
			\Pi_1(s)
			\|w(s)\|_{H_0^s(\Omega)}^{2}\,ds-
			\int_{\tau}^{T}
			\int_{\Omega}
			\bigl(f(u^{1})-f(u^{2})\bigr)
			w_t\,dx\,ds.
		\end{align}
		Next, multiplying \eqref{difference-equation} by $\Pi_1(t)w$,
		integrating over $\Omega\times(\tau,T)$, and integrating by parts in
		time, we obtain
		\begin{equation}\label{asymp-3}
			\int_\tau^T
			\Pi_1(s)\|w(s)\|_{H_0^s(\Omega)}^{2}\,ds
			=
			\int_\tau^T
			\Pi_1(s)\|w_t(s)\|^{2}\,ds
			+\sum_{j=1}^{5}J_j,
		\end{equation}
		where
		\begin{eqnarray*}
			J_1
			&=&
			-\frac{\gamma}{2}
			\int_\tau^T
			\Pi_1(s)^2
			\int_\Omega
			w_tw\,dx\,ds,\\
			J_2
			&=&
			-\frac{\gamma}{2}
			\int_\tau^T
			\Pi_1(s)\Pi_2(s)
			\int_\Omega
			\bigl(u_t^1+u_t^2\bigr)w\,dx\,ds,\\
			J_3
			&=&
			-\int_\tau^T
			\Pi_1(s)
			\int_\Omega
			\bigl(f(u^1)-f(u^2)\bigr)w\,dx\,ds,\\
			J_4
			&=&
			-\Biggl[
			\Pi_1(s)
			\int_\Omega
			w_tw\,dx
			\Biggr]_\tau^T,\\
			J_5
			&=&
			\int_\tau^T
			\Pi_1'(s)
			\int_\Omega
			w_tw\,dx\,ds.
		\end{eqnarray*}
		Substituting \eqref{asymp-3} into \eqref{difference-energy-2}, we obtain
		\begin{align}\label{difference-energy-3}
			&\frac12\|W(T)\|_{\mathcal H_s}^{2}
			-\frac12\|W(\tau)\|_{\mathcal H_s}^{2}
			+\frac{\gamma}{2}
			\left(
			1-\frac{\max\{1,q\}}{2}
			\right)
			\int_{\tau}^{T}
			\Pi_1(s)\|w_t(s)\|^{2}\,ds
			\nonumber\\
			&\le
			-\int_{\tau}^{T}
			\int_{\Omega}
			\bigl(f(u^{1})-f(u^{2})\bigr)
			w_t\,dx\,ds
			+\frac{\gamma\max\{1,q\}}{4}
			\sum_{j=1}^{5}J_j .
		\end{align}
		Since $\max\{1,q\}\leq 2$ for $q\in[1/2,2]$, it follows that
		$1-\frac{\max\{1,q\}}{2}\geq 0$
		and
		$\frac{\gamma\max\{1,q\}}{4}\leq \frac{\gamma}{2}$.
		Hence, inequality \eqref{difference-energy-3} gives
		\begin{align}\label{difference-energy-4}
			\|W(T)\|_{\mathcal H_s}^{2}
			\le
			\|W(\tau)\|_{\mathcal H_s}^{2}
			-
			2\int_{\tau}^{T}
			\int_{\Omega}
			\bigl(f(u^{1})-f(u^{2})\bigr)
			w_t\,dx\,ds+\gamma\sum_{j=1}^{5}J_j.
		\end{align}
		Integrating \eqref{difference-energy-4} with respect to $\tau$ over $(0,T)$, we obtain
		\begin{align}\label{difference-energy-5}
			T\|W(T)\|_{\mathcal H_s}^{2}
			\le&\;
			\int_{0}^{T}
			\|W(\tau)\|_{\mathcal H_s}^{2}\,d\tau
			-2
			\int_{0}^{T}
			\int_{\tau}^{T}
			\int_{\Omega}
			\bigl(f(u^{1})-f(u^{2})\bigr)
			w_t\,dx\,ds\,d\tau
			\nonumber\\
			&\qquad
			+\gamma
			\int_{0}^{T}
			\sum_{j=1}^{5}J_j,d\tau.
		\end{align}
		We next estimate the terms $J_1,\ldots,J_5$. Since
		$
		|\Pi_j(t)|\le C_B,
		$
		for $j=1,2$, it follows immediately that
		$$
		J_1,\;J_2\le C_B\int_{\tau}^T\|w(s)\|ds.
		$$
		Moreover, Assumption~\eqref{assumption_f'}, together with H\"older's inequality, implies
		$$
		J_3
		\le
		C_B\int_{\tau}^T\|w(s)\|_{p+2}^{2}ds,
		$$
		whereas
		$$
		J_4
		\le
		C_B
		\sup_{t\in[0,T]}
		\|w(t)\|.
		$$
		Finally, differentiating $\Pi_1(t)$ and using Assumption \eqref{assumption_f'} and applying H\"older's inequality with
		$
		\frac{p}{2(p+1)}+\frac{1}{2(p+1)}+\frac{1}{2}=1,
		$
		together with the continuous embedding
		$H_0^s(\Omega)\hookrightarrow L^{2(p+1)}(\Omega)$, we obtain
		\begin{align*}
			|\Pi_1'(t)|
			&=
			\left|
			q\sum_{i=1}^{2}
			\|U^i(t)\|_{\mathcal H_s}^{2(q-1)}
			\frac{d}{dt}
			\|U^i(t)\|_{\mathcal H_s}^{2}
			\right|
			\\
			&\le
			q\sum_{i=1}^{2}
			\|U^i(t)\|_{\mathcal H_s}^{2(q-1)}
			\left(
			\gamma
			\|U^i(t)\|_{\mathcal H_s}^{2q}
			\|u_t^i(t)\|^{2}
			+
			\left|
			\int_{\Omega}f(u^i)u_t^i\,dx
			\right|
			+
			\|h\|\,
			\|u_t^i(t)\|
			\right)
			\\
			&\le
			C_B
			q
			\sum_{i=1}^{2}
			\|U^i(t)\|_{\mathcal H_s}^{2q-1}.
		\end{align*}
		Since $q\ge\frac12$, we have $2q-1\ge0$. Therefore,
		$\|U^i(t)\|_{\mathcal H_s}^{2q-1}\le C_B$
		along bounded trajectories, which yields
		$$
		J_5\le C_B\int_{\tau}^T\|w(s)\|ds.
		$$
		Substituting the above estimates into \eqref{difference-energy-5} and using that $L^{p+2}(\Omega)\hookrightarrow L^2(\Omega)$, we conclude that
		\begin{align}\label{difference-energy-6}
			T\|W(T)\|_{\mathcal H_s}^{2}
			\le&\;
			\int_{0}^{T}
			\|W(\tau)\|_{\mathcal H_s}^{2}\,d\tau
			+C_{B,T}
			\sup_{t\in[0,T]}
			\|w(t)\|_{p+2}\nonumber\\
			&
			+2\left|
			\int_{0}^{T}
			\int_{\tau}^{T}
			\int_{\Omega}
			\bigl(f(u^{1})-f(u^{2})\bigr)
			w_t\,dx\,ds\,d\tau
			\right|.
		\end{align}
		Now, multiplying equation \eqref{P} by $w$ and integrating over $\Omega\times(\tau,T)$, we obtain
		\begin{align*}
			\int_0^T\|W(\tau)\|^2_{\mathcal{H}_s}d\tau=2\int_0^T\|w_t(\tau)\|^2d\tau+\sum_{j=1}^{4}L_j,
		\end{align*}
		where
		\begin{align*}
			L_1
			&=
			-\frac{\gamma}{2}
			\int_\tau^T
			\Pi_1(s)
			\int_\Omega
			w_tw\,dx\,ds,\\
			L_2
			&=
			-\frac{\gamma}{2}
			\int_\tau^T
			\Pi_2(s)
			\int_\Omega
			\bigl(u_t^1+u_t^2\bigr)w\,dx\,ds,\\
			L_3
			&=
			-\int_\tau^T\int_\Omega
			\bigl(f(u^1)-f(u^2)\bigr)w\,dx\,ds,\\
			L_4
			&=
			-\Biggl[\int_\Omega
			w_tw\,dx
			\Biggr]_\tau^T.
		\end{align*}
		Using arguments similar to those used to estimate the terms $J_j$, we have
		\begin{eqnarray}\label{asymp-b}
			\int_0^T\|W(\tau)\|^2_{\mathcal{H}_s}d\tau
			\le
			2\int_0^T\|w_t(\tau)\|^2\,d\tau
			+
			C_{B,T}\sup_{t\in[0,T]}\|w(t)\|_{p+2}.
		\end{eqnarray}
		Substituting \eqref{asymp-b} into \eqref{difference-energy-6}, we obtain
		\begin{eqnarray}\label{difference-energy-8}
			T\|W(T)\|_{\mathcal H_s}^{2}
			&\le&
			2\int_{0}^{T}
			\|w_t(\tau)\|^{2}\,d\tau
			++
			C_{B,T}
			\sup_{t\in[0,T]}
			\|w(t)\|_{p+2}\nonumber\\
			&&+
			2\left|
			\int_{0}^{T}
			\int_{\tau}^{T}
			\int_{\Omega}
			\bigl(f(u^{1})-f(u^{2})\bigr)
			w_tdxdsd\tau
			\right|.
		\end{eqnarray}
		Using \eqref{global-in-time}, we have
		\begin{eqnarray}\label{asymp-c}
			2\int_{0}^{T}
			\|w_t\|^{2}\,d\tau\le2T^{q/q+1}\left(\int_0^T\|w_t\|^{2(q+1)}d\tau\right)^{1/q+1}\le T^{q/q+1}C_B.
		\end{eqnarray}
		Replacing \eqref{asymp-c} in \eqref{difference-energy-8}, we obtain
		\begin{eqnarray}\label{asymp-d}
			\|W(T)\|_{\mathcal H_s}^{2}
			&\le&\frac{C_B}{T^{1/q+1}}+
			C_{B,T}
			\sup_{t\in[0,T]}
			\|w(t)\|_{p+2}\nonumber\\
			&&+
			\frac{2}{T}\left|
			\int_{0}^{T}
			\int_{\tau}^{T}
			\int_{\Omega}
			\bigl(f(u^{1})-f(u^{2})\bigr)
			w_tdxdsd\tau
			\right|.
		\end{eqnarray}
		Using $W(t)=S_tU_0^1-S_tU_0^2$, inequality \eqref{inequality-main} follows from \eqref{asymp-d}, which completes the proof of Lemma~\ref{lemma-stability}.
	\end{proof}

	\begin{proposition}\label{prop-asymptotic}
		Under the assumptions of Lemma~\ref{lemma-stability}, the dynamical system $(\mathcal{H}_s,S_t)$ is asymptotically smooth.
	\end{proposition}
	\begin{proof}
		Given $\epsilon>0$, choose $T=T_{\epsilon,B}>0$ sufficiently large such that
		\begin{equation}\label{asymp-e}
			\frac{C_B}{T^{1/q+1}}<\epsilon.
		\end{equation}
		Define the functional
		$\Psi_T:\mathcal{H}_s\times\mathcal{H}_s\to\mathbb{R}$ by
		\begin{align*}
			\Psi_T\big(U_0^1,U_0^2\big)
			:=
			C_{B,T}
			\sup_{t\in[0,T]}
			\|w(t)\|_{p+2}+
			\frac{2}{T}
			\left|
			\int_0^T
			\int_{\tau}^T
			\int_\Omega
			\bigl(f(u^1)-f(u^2)\bigr)w_t
			dxdsd\tau
			\right|,
		\end{align*}
		where $w(t)=u^1(t)-u^2(t)$. It follows from \eqref{asymp-d} and \eqref{asymp-e} that
		\begin{align*}
			\|S_TU_0^1-S_TU_0^2\|_{\mathcal{H}_s}
			\leq
			\epsilon+
			\Psi_T\big(U_0^1,U_0^2\big),
		\end{align*}
		for all $U_0^1,U_0^2\in B$.
		
		To complete the proof, it remains to show that $\Psi_{\epsilon,B,T}$ is a contractive function. Indeed, let $\{U_0^n\}_{n\in\mathbb{N}}\subset B$. We claim that
		\begin{align*}
			\lim\inf_{n\to\infty}\lim\inf_{m\to\infty}
			\Psi_{\epsilon,B,T}(U_0^n,U_0^m)=0.
		\end{align*}
		Indeed, first, since the embedding
		$H_0^s(\Omega)\hookrightarrow L^{p+2}(\Omega)
		$
		is compact, it follows directly that
		\begin{align*}
			\lim\inf_{n\to\infty}\lim\inf_{m\to\infty}
			\left\{
			C_{B,T}
			\sup_{t\in[0,T]}
			\|u^m(t)-u^n(t)\|_{p+2}
			\right\}
			=0.
		\end{align*}
		Moreover, combining the boundedness and positive invariance of $B$ with the Aubin--Lions compactness theorem and \cite[Step 2, Lemma 3.38]{chueshov-white}, we obtain
		\begin{align*}
			\lim\inf_{n\to\infty}\lim\inf_{m\to\infty}
			\left\{
			\frac{2}{T}
			\left|
			\int_0^T
			\int_{\tau}^T
			\int_\Omega
			\bigl(f(u^m)-f(u^n)\bigr)
			\bigl(u_t^m-u_t^n\bigr)
			dxdsd\tau
			\right|
			\right\}
			=0.
		\end{align*}
		Consequently, $\Psi_{\epsilon,B,T}$ is a contractive function. Therefore, the asymptotic smoothness of the dynamical system follows from Theorem~\ref{theo-asymp-smootness}.
		
	\end{proof}

	\begin{proposition}\label{Prop-stationary-set}
		Let $\Phi$ be the Lyapunov function given in Proposition \ref{prop-gradient-system}. Then $\Phi(U)$ is bounded from above on any bounded subset of $\mathcal{H}_s$ and the set $\Phi_R=\{U:\Phi(U)\le R\}$ is bounded for every $R$.
	\end{proposition}
	\begin{proof}
		Let $B$ be bounded set in $\mathcal{H}_s$. Let $U(t)=S_tU_0$ be a mild solution of the problem (\ref{P}) such that $U_0\in B$.
		Multiplying the equation \eqref{P} by $u_t$ and integrating $\Omega\times [0,t]$, we obtain
		\begin{align*}\Phi(U(t))+\int_0^t\|U\|^{2q}_{\mathcal{H}_s}\|u_t\|^2ds=\Phi(U(0)),\quad t\ge 0.
		\end{align*}
		Since $\Phi(U_0)=E(U_0)\le C_B$, it follows directly from the above identity that $\Phi(U)=E(U)$ is bounded from above on any bounded subset $B$ of $\mathcal{H}_s$.
		
		Now, we define $\Phi_R=\{U\in \mathcal{H}_s:\Phi(U)\le R\}$. From \eqref{global-in-time1}
		\begin{align*}
			||U(t)||^2_{\mathcal{H}_s}\le\frac{4}{\kappa}\left[\kappa_0+\Phi(U(t))\right]\le \frac{4}{\kappa}\left[\kappa_0+R\right]=:C_R,
		\end{align*}
		then $\Phi_R$ is bounded in $\mathcal{H}_s$.
	\end{proof}
	
	\begin{proposition}\label{Prop-stationary-set-2}
		The set of stationary solutions of problem \eqref{P} is bounded in
		$\mathcal{H}_s$.
	\end{proposition}
	\begin{proof}
		Let $u$ be a stationary solution of problem \eqref{P}. Then $u$
		satisfies
		$$
		(-\Delta)^s u+f(u)=h.
		$$
		Testing the above equation with $u$, we obtain
		$$
		\|u\|_{H_0^s(\Omega)}^2
		=
		\int_\Omega h u\,dx
		-
		\int_\Omega f(u)u\,dx.
		$$
		By assumption \eqref{assumption_f} and Poincar\'e inequality, we have
		\begin{align*}
			\|u\|_{H_0^s(\Omega)}^2
			&\le
			\|h\|\|u\|
			+C_f|\Omega|
			+\frac{3}{4}c_f\|u\|^2\\
			&\le
			K\|h\|\|u\|_{H_0^s(\Omega)}
			+\frac{3}{4}K^2c_f\|u\|_{H_0^s(\Omega)}^2
			+C_f|\Omega|.
		\end{align*}
		Applying Young's inequality,
		$$
		K\|h\|\|u\|_{H_0^s(\Omega)}
		\le
		2K^2\|h\|^2
		+\frac{1}{8}\|u\|_{H_0^s(\Omega)}^2,
		$$
		we deduce that
		$$
		\left(
		1-\frac{1}{8}-\frac{3}{4}K^2c_f
		\right)
		\|u\|_{H_0^s(\Omega)}^2
		\le
		2K^2\|h\|^2
		+C_f|\Omega|.
		$$
		Since $c_f<1/K^2$, we have
		$$
		1-\frac{1}{8}-\frac{3}{4}K^2c_f
		>
		\frac{1}{8},
		$$
		and therefore
		\[
		\|u\|_{H_0^s(\Omega)}^2
		\le
		8\left(
		2K^2\|h\|^2
		+C_f|\Omega|
		\right).
		\]
		Since every stationary solution is of the form $(u,0)$, it follows that
		the set $\mathcal{N}$ is bounded in $\mathcal{H}_s$.
	\end{proof}
	
	\paragraph{Conflict of interest} On behalf of all authors, the corresponding author states that there is no conflict of interest. 
	
	\paragraph{Author contributions}
	All authors contributed to the study conception and design. 
	
	\paragraph{Data availability} 
	Data sharing not applicable to this article as no datasets were generated or analysed during the current study.

\end{document}